\documentclass[10pt]{amsart}
\usepackage{amsmath,amssymb,amsthm,mathrsfs,xspace}
\usepackage[shortlabels]{enumitem}
\title{A note on relative amenability}
\author[Phillip Wesolek]{Phillip Wesolek}

\address{MSCS University of Illinois at Chicago,
   322 Science and Engineering Offices MC249,
   851 S. Morgan St. Chicago, Il 60607-7045,
   USA}

\address{ Current address:
	Universit\'{e} catholique de Louvain,
	Institut de Recherche en Math\'{e}matiques et Physique (IRMP),
	Chemin du Cyclotron 2, box L7.01.02,
	1348 Louvain-la-Neuve, Belgium}
\email{phillip.wesolek@uclouvain.be}

\date{June 2014}

\newtheorem{thm}{Theorem}[section]

\newtheorem{prop}[thm]{Proposition}
\newtheorem{lem}[thm]{Lemma}
\newtheorem{cor}[thm]{Corollary}
\theoremstyle{definition}
\newtheorem{defn}[thm]{Definition}
\newtheorem{rmk}[thm]{Remark}
\newtheorem{quest}[thm]{Question}

\newtheorem*{ack*}{Acknowledgments}

\newcommand{\U}{\mathcal{U}}

\newcommand{\N}{\mathbb{N}}
\newcommand{\normal}{\trianglelefteq}

\newcommand{\X}{\mathscr{X}}
\newcommand{\ms}[1]{\mathscr{#1}}

\newcommand{\Es}{\mathscr{E}^*}
\newcommand{\E}{\mathscr{E}}
\newcommand{\Y}{\mathscr{Y}}
\newcommand{\Ys}{\mathscr{Y}^*}
\newcommand{\Rad}{\mathrm{Rad}}

\newcommand{\sleq}{\leqslant}

\newcommand{\tdlc}{t.d.l.c.\@\xspace}
\newcommand{\tdlcsc}{t.d.l.c.s.c.\@\xspace}
\newcommand{\lc}{l.c.\@\xspace}
\newcommand{\rk}{\mathop{\rm rk}}
\newcommand{\ord}{\mathop{\rm ORD}}
\newcommand{\Res}[1]{\mathop{\rm Res}_{#1}\nolimits}

\begin{document}

\begin{abstract} 
P-E. Caprace and N. Monod isolate the class $\ms{X}$ of locally compact groups for which relatively amenable closed subgroups are amenable. It is unknown if $\X$ is closed under group extension. In this note, we exhibit a large, group extension stable subclass of $\X$, which suggests $\X$ is indeed closed under group extension. Along the way, we produce generalizations of the class of elementary groups and obtain information on groups outside $\X$.
\end{abstract}

\maketitle

\section{Introduction}
In \cite{CM14}, P-E. Caprace and N. Monod introduce the notion of relative amenability:

\begin{defn}
For a locally compact group $G$, a closed subgroup $H\sleq G$ is \textbf{relatively amenable} if $H$ fixes a point in every non-empty convex compact $G$-space.
\end{defn}
\noindent A convex compact $G$-space is a convex compact subset of a locally convex topological vector space such that the subset has a continuous affine $G$ action.

They go on to study the relationship between amenability and relative amenability. In particular, they isolate a large and interesting class of locally compact groups.

\begin{defn} 
The class $\X$ is the collection of locally compact groups for which every relatively amenable closed subgroup is amenable. 
\end{defn}

\begin{thm}[Caprace--Monod, {\cite[Theorem 2]{CM14}}]\label{thm:closure}
\
\begin{enumerate}[(a)]
\item $\X$ contains all discrete groups.
\item $\X$ contains all groups amenable at infinity.
\item $\X$ is closed under taking closed subgroups.
\item $\X$ is closed under taking (finite) direct products.
\item $\X$ is closed under taking adelic products.
\item $\X$ is closed under taking directed unions of open subgroups.
\end{enumerate}
Let $N\trianglelefteq G$ be a closed normal subgroup of a locally compact group $G$.
\begin{enumerate}[resume*]
\item If $N$ is amenable, then $G\in\X \Longleftrightarrow G/N\in\X$.
\item If $N$ is connected, then $G\in\X \Longleftrightarrow G/N\in\X$.
\item If $N$ is open, then $G\in\X \Longleftrightarrow N\in\X$.
\item If $N$ is discrete and $G/N\in\X$, then $G\in\X$.
\item If $N$ is amenable at infinity and $G/N\in\X$, then $G\in\X$.
\end{enumerate}
\end{thm}

Two questions concerning the class $\X$ arise:
\begin{quest}[Caprace--Monod]\label{quest:ext} 
\
\begin{enumerate}[(1)]
\item Is $\X$ stable under group extension?
\item Are there locally compact groups outside of $\X$?
\end{enumerate}
\end{quest}

Our note contributes primarily to the study of the former. Indeed, let $\mathscr{Y}$ be the smallest collection of locally compact groups such that
\begin{enumerate}[(1)]
\item $\Y$ contains all compact groups, discrete groups, and connected groups,

\item $\Y$ is closed under group extensions, and

\item $\Y$ is closed under directed unions of open subgroups. That is to say, if $G=\bigcup_{i\in I}O_i$ where $\{O_i\}_{i\in I}$ is a directed system of open subgroups of $G$ such that $O_i\in \Y$ for each $i$, then $G\in \Y$.

\end{enumerate}

\noindent We prove the following: 

\begin{thm}\label{thm:Y}
The class $\Y$ is contained in $\X$ and enjoys the following additional permanence properties: 
\begin{enumerate}[(a)]
\item $\Y$ is closed under taking closed subgroups.

\item $\Y$ is closed under taking quotients by closed normal subgroups. 
\end{enumerate}
\end{thm}

This provides evidence the class $\X$ may be stable under group extensions. Furthermore, it gives information on how groups $G \not \in \X$ look like; see Remark~\rm\ref{rem:counterexample} below. 

\begin{ack*}
The author would like to thank Pierre-Emmanuel Caprace for his many helpful suggestions and for supporting a week research stay at the Universit\'{e} catholique de Louvain during which this note was partially developed. This note was further developed during a visit to the Fields Institute; the author thanks the Institute for its hospitality.
\end{ack*}

\section{Generalities on locally compact groups}

All groups are taken to be Hausdorff topological groups. We abbreviate ``locally compact" by ``l.c.", ``totally disconnected" by ``t.d.", and ``second countable" by ``s.c.". We write $H\sleq _oG$ to indicate $H$ is an open subgroup of $G$. We denote by $\U(G)$ the collection of compact open subgroups of $G$. 

As its statement is not the obvious generalization, we recall the first isomorphism theorem for locally compact groups (see \cite[(5.33)]{HR79}):

\medskip 
\textit{
Let $G$ be a locally group, $A\sleq G$ be closed, and $H\trianglelefteq G$ be closed. If $A$ is $\sigma$-compact  and $AH$ is closed, then $AH/H\simeq A/A\cap H$ as topological groups.}

\medskip 
Locally compact $\sigma$-compact groups are close to being second countable.

\begin{thm}[{Kakutani--Kodaira, see \cite[(8.7)]{HR79}}]\label{KK}
If $G$ is a $\sigma$-compact \tdlc group, then there is a compact $K\trianglelefteq G$ such that $G/K$ is metrizable, hence second countable. 
\end{thm}

A \tdlc group $G$ is said to be a \textbf{small invariant neighborhood group}, denoted SIN, if $G$ admits a basis at $1$ of compact open normal subgroups. These groups when compactly generated admit a useful characterization.

\begin{thm}[Caprace--Monod, {\cite[Corollary~4.1]{CM11}}]\label{CM_SIN}
A compactly generated \tdlc group is SIN if and only if it is residually discrete.
\end{thm}

The \textbf{discrete residual} of a \tdlc group $G$, denoted $\Res{}(G)$, is the intersection of all open normal subgroups. When $G$ is compactly generated, Theorem~\rm\ref{CM_SIN} implies $G/\Res{}(G)$ is a SIN group.

Groups ``built by hand" from profinite and discrete groups often appear when studying \tdlc groups. The class of elementary groups captures the intuitive idea of such groups.
\begin{defn}
The class of \textbf{elementary groups} is the smallest class $\E$ of \tdlcsc groups such that
\begin{enumerate}[(1)]

\item $\E$ contains all second countable profinite groups and countable discrete groups.

\item $\E$ is closed under taking group extensions of second countable profinite or countable discrete groups. I.e. if $G$ is a \tdlcsc group and $H\trianglelefteq G$ is a closed normal subgroup with $H\in \E$ and $G/H$ profinite or discrete, then $G\in \E$.

\item If $G$ is a \tdlcsc group and $G=\bigcup_{i\in \N}O_i$ where $(O_i)_{i\in \N}$ is an $\subseteq$-increasing sequence of open subgroups of $G$ with $O_i\in\E$ for each $i$, then $G\in\E$. We say $\E$ is \textbf{closed under countable increasing unions}.
\end{enumerate}
\end{defn}

The class $\E$ enjoys robust permanence properties, which supports the thesis $\E$ is exactly the groups ``built by hand."

\begin{thm}[{\cite[Theorem 3.18, Theorem 5.7]{W_1_14}}]\label{thm:closure_main}
$\E$ enjoys the following permanence properties:
\begin{enumerate}[(a)]
\item $\E$ is closed under group extension.

\item If $G\in \E$, $H$ is a \tdlcsc group, and $\psi:H\rightarrow G$ is a continuous, injective homomorphism, then $H\in \E$. In particular, $\E$ is closed under taking closed subgroups.

\item $\E$ is closed under taking quotients by closed normal subgroups.

\item If $G\in \E$, $H$ is a \tdlcsc group, and $\psi:G\rightarrow H$ is a continuous, injective homomorphism with dense normal image, then $H\in \E$.
\end{enumerate}
\end{thm}

The class $\E$ admits a canonical rank function:

\begin{defn}
The \textbf{decomposition rank} $\xi:\E\rightarrow [1,\omega_1)$ is an ordinal-valued function satisfying the following properties:
	\begin{enumerate}[(a)]
		\item $\xi(G)=1$ if and only if $G \simeq \{1\}$;
		
		\item If $G \in \E$ is non-trivial and $G=\bigcup_{i\in \N}O_i$ with $(O_i)_{i\in \N}$ some $\subseteq$-increasing sequence of compactly generated open subgroups of $G$, then 
		\[
		\xi(G)=\sup_{i\in \N}\xi(\Res{}(O_i))+1.
		\]
	\end{enumerate}
\end{defn}

By \cite[Theorem 4.7]{W_1_14} and \cite[Lemma 4.10]{W_1_14}, such a function $\xi$ exists, is unique, and is equivalent to the decomposition rank as it is defined in \cite{W_1_14}. 

The decomposition rank has a useful permanence property.

\begin{lem}[{\cite[Lemma 2.9]{RW_LC_16}}]\label{lem:compact_ext}
Suppose $G$ is a \tdlcsc group and $N$ is a closed cocompact normal subgroup of $G$. If $N\in \E$, then $G\in \E$ with $\xi(G)=\xi(N)$.
\end{lem}

\section{The class $\Es$}

We first argue $\E$ is a subclass of $\X$. To this end, we require a lemma due to Caprace and Monod.

\begin{lem}[{\cite[Lemma 13]{CM14}}]\label{lem:grp_ext}
Suppose $G$ is a locally compact group, $H\trianglelefteq G$ is closed, and $O\sleq_o G$ contains $H$. If $O\in \X$ and $G/H\in \X$, then $G\in \X$.
\end{lem}

\begin{thm}\label{key lem}
$\E\subseteq \X$.  
\end{thm}
\begin{proof}We argue by induction on $\xi(H)$ for $H\in \X$. As the base case is obvious, suppose $\xi(H)=\beta+1$ and write $H=\bigcup_{i\in \N}O_i$ with $(O_i)_{i\in \N}$ an inclusion increasing sequence of compactly generated open subgroups of $H$.
	
Fixing $i\in \N$, the quotient $O_i/\Res{}(O_i)$ is a SIN group via Theorem~\ref{CM_SIN}. It is therefore compact-by-discrete, and Theorem~\ref{thm:closure} implies it is in $\X$. On the other hand, the definition of the decomposition rank gives that $\xi(\Res{}(O_i))\leq \beta$, so taking a compact open subgroup $V\leq O_i$, the group $V\Res{}(O_i)$ is elementary with rank at most $\beta $ via Lemma~\ref{lem:compact_ext}. Applying the induction hypothesis, we deduce that $V\Res{}(O_i)\in \X$. Lemma~\ref{lem:grp_ext} now ensures that $O_i\in \X$, and since $\X$ is closed under direct limits, we conclude that $H\in \X$, completing the induction.
\end{proof}

We next relax the second countability assumption on $\mathscr{E}$ by introducing the following class:

\begin{defn}
The class $\mathscr{E}^*$ is the smallest collection of \tdlc groups such that
\begin{enumerate}[(1)]
\item $\Es$ contains all profinite groups and discrete groups,

\item $\Es$ is closed under group extensions of profinite and discrete groups, and

\item $\Es$ is closed under directed unions of open subgroups.

\end{enumerate}
\end{defn}
\noindent There is an ordinal rank on $\Es$. For $G\in \mathscr{E}$, define

\begin{enumerate}[$\bullet$]
\item $G\in \Es_0$ if and only if $G$ is profinite or discrete.

\item Suppose $\Es_{\alpha}$ is defined. Put $G\in (\Es)_{\alpha}^e$ if and only if there exists $N\trianglelefteq G$ such that $N\in \Es_{\alpha}$ and $G/N\in \Es_{0}$. Put $G\in (\Es)_{\alpha}^l$ if and only if $G=\bigcup_{i\in I}H_i$ where $(H_i)_{i\in I}$ is an $\subseteq$-directed set of open subgroups of $G$ and $H_i\in \Es_{\alpha}$ for each $i\in I$. Define $\Es_{\alpha+1}:=(\Es)_{\alpha}^e\cup (\Es)_{\alpha}^l$.

\item For $\lambda$ a limit ordinal, $\Es_{\lambda}:=\bigcup_{\beta<\lambda}\Es_{\beta}$.
\end{enumerate}

Certainly, $\Es=\bigcup_{\alpha\in \ord} \Es_{\alpha}$, so for $G\in \Es$, we define
\[
\rk(G):=\min\{\alpha\in\ord\;|\;G\in \Es_{\alpha}\}.
\]
We call $\rk(G)$ the \textbf{construction rank} of $G$.

The construction rank has a number of nice properties.
\begin{lem} Let $G\in \Es$. Then
\begin{enumerate}[(a)]
\item $\rk(G)$ is a successor ordinal when $\rk(G)$ is non-zero.

\item If $G$ is compactly generated and has non-zero rank, then $\rk(G)$ is given by a group extension. I.e.\ if $\rk(G)=\beta+1$, there is $H\trianglelefteq G$ such that $\rk(H)=\beta$ and $\rk(G/H)=0$.

\item If $O\sleq_oG$, then $O\in \Es$ and $\rk(O)\sleq \rk(G)$.
\end{enumerate}
\end{lem}

\begin{proof} 
These follow by transfinite induction on $\rk(G)$.
\end{proof}

\begin{prop}\label{Es E}
Let $G\in \Es$ be  $\sigma$-compact. If $K\trianglelefteq G$ is compact with $G/K$ second countable, then $G/K\in \E$.
\end{prop}
\begin{proof} We induct on the construction rank of $G$. As the proposition is immediate if $\rk(G)=0$, suppose $\rk(G)=\alpha +1$. 

Suppose $\rk(G)$ is given by a directed union, so there is $(O_i)_{i\in I}$ a directed system of open subgroups of $G$ such that $G=\bigcup_{i\in I}O_i$ with $\rk(O_i)\sleq \alpha$ for each $i\in I$. Certainly, 
\[
G/K=\bigcup_{i\in I}O_iK/K,
\]
and we may find a countable subcover $(O_iK/K)_{i\in \N}$ since $G/K$ is Lindel\"{o}f. One checks this cover may be taken to be an increasing $\subseteq$-chain. Further, $O_iK/K\simeq O_i/(O_i\cap K)$, $\rk(O_i)\sleq \alpha$, and $O_i\cap K$ is a compact normal subgroup of $O_i$ whose quotient is second countable. The induction hypothesis implies $O_i/(O_i\cap K)\in \E$, and as $G/K$ is the countable increasing union of $(O_iK/K)_{i\in \N}$, we conclude $G/K\in \E$.

Suppose the construction rank of $G$ is given by a group extension; say $H\trianglelefteq G$ is such that $\rk(H)=\alpha$ and $\rk(G/H)=0$. Since $H$ is $\sigma$-compact and $K$ compact, we have $HK/K\simeq H/(H\cap K)$, and as $\rk(H)= \alpha$, the induction hypothesis implies $H/(H\cap K)\in \E$. On the other hand, $HK/H\trianglelefteq G/H$ is closed, and the quotient 
\[
(G/H)/(HK/H)\simeq G/HK\simeq (G/K)/(HK/K)
\] 
is second countable and either discrete or compact. We conclude $G/K$ is a group extension of $G/HK$ by $HK/K$, hence $G/K\in \E$. This completes the induction, and we conclude the proposition.
\end{proof}

\begin{thm}\label{thm:es_x} 
$\Es\subseteq \X$. 
\end{thm}
\begin{proof}
Take $G \in \Es$. Every locally compact group is the directed union of its open compactly generated subgroups. Since $\X$ is closed under directed unions, we may assume $G$ is compactly generated and therefore $\sigma$-compact. Applying Theorem~\rm\ref{KK}, there is $K\trianglelefteq G$ such that $G/K$ is second countable. Proposition~\rm\ref{Es E} implies $G/K\in \E$, and by Theorem~\rm\ref{key lem}, we have that $G/K\in \X$. Since $K$ is compact, we deduce from Theorem~\rm\rm\ref{thm:closure} that $G\in \X$.
\end{proof} 

We conclude this section by noting three permanence properties of $\Es$.

\begin{thm}\label{thm:es_closure}
$\Es$ enjoys the following permanence properties:
\begin{enumerate}[(a)]
\item $\Es$ is closed under group extension.

\item If $G\in \Es$, then every \tdlc group admitting a  continuous, injective homomorphism into $G$ also belongs to $ \Es$. In particular, $\Es$ is closed under taking closed subgroups.

\item $\Es$ is closed under taking quotients by closed normal subgroups.
\end{enumerate}
\end{thm}

\begin{proof}
Claim $(a)$ follows just in the case of $\E$; see the proof of \cite[Proposition 3.5]{W_1_14}.

\medskip \noindent
For $(b)$, let $G\in \Es$, let $H$ be a \tdlc group, and let $\psi \colon H \to G$ be a continuous, injective homomorphism.  Every locally compact group is the directed union of its compactly generated open subgroups. Since $\Es$ is closed under directed unions, we may therefore assume $H$ is compactly generated. Therefore, $P := \overline{\psi(H)}$ is also compactly generated. 

Since $H$ is $\sigma$-compact, it has a compact normal subgroup $K$ such that $H/K$ is second countable; likewise, $P$ admits a compact normal subgroup $L\normal P$ so that $P/L$ is second countable. The subgroup $L\psi(K)$ is then a compact normal subgroup of $P$ so that $P/L\psi(K)$ is second countable. The group $H/\psi^{-1}(L)K$ is also second countable, and $\psi$ induces an injective, continuous homomorphism $\tilde{\psi}:H/\psi^{-1}(L)K\rightarrow P/L\psi(K)$.  Applying Proposition~\ref{Es E}, we conclude $P/L\psi(K) \in \E$ and, via Theorem~\rm\ref{thm:closure_main}, $H/\psi^{-1}(L)K \in \E$. 

On the other hand, we have an injection $\psi^{-1}(L)K\rightarrow \psi(K)L$ with the latter group residually discrete. The group $\psi^{-1}(L)K$ is then residually discrete, and it follows from Theorem~\rm\ref{CM_SIN} that $\psi^{-1}(L)K\in \Es$. In view of part $(a)$, we deduce $H \in \Es$ as desired. 

\medskip \noindent
For $(c)$, suppose $L\trianglelefteq G$ is closed; as above, it suffices to consider the case that $G$ is compactly generated. In view of Proposition~\rm\ref{Es E}, there is $K\trianglelefteq G$ compact such that $G/K\in \E$. The group $LK/K$ is a closed normal subgroup of $G/K$, so Theorem~\rm\ref{thm:closure_main} implies 
\[
(G/L)/(LK/L)\simeq G/LK \simeq (G/K)/(LK/K) \in \E
\]
On the other hand, $LK/L\simeq K/(K\cap L)$ with the latter group compact. We conclude $G/L$ is compact-by-$\E$ and, via $(a)$, belongs to $\Es$.
\end{proof}

\section{The class $\Y$}

In order to prove Theorem~\rm\ref{thm:Y}, we introduce the class  $\Ys$  consisting of those locally compact groups $G$ such that $G/G^\circ \in \Es$ where $G^{\circ}$ is the connected component of the identity. By definition, we have that $\Es \subseteq \Y$, hence $\Ys \subseteq \Y$ since $\Y$ is stable under group extensions and contains all connected \lc groups. We will eventually show that $\Ys = \Y$. The proof of this relies on the following. 

\begin{prop}\label{prop:Ys}
The class $\Ys$ enjoys the following permanence properties:
\begin{enumerate}[(a)]
\item $\Ys$ is closed under directed unions of open subgroups. 

\item $\Ys$ is closed under taking closed subgroups. 

\item $\Ys$ is closed under taking group extensions. 

\item $\Ys$ is closed under taking quotients by closed normal subgroups. 
\end{enumerate}
\end{prop}

The proof of (c)  requires the following subsidiary fact. 

\begin{lem}\label{quot lem}
Let $G$ be a $\sigma$-compact \lc~group with closed $L,P\trianglelefteq G$ so that $L \geqslant G^\circ$.  If $P\in \Es$, then $\overline{PL}/L\in \Es$.
\end{lem}

\begin{proof} Set $H:=\overline{PL}/L$ and let $\pi:P\rightarrow H$ be the restriction of the usual projection. The image of $\pi$ is a dense normal subgroup of $H$. Theorem~\ref{thm:es_closure} ensures $\Es$ is closed under taking quotients, so we may assume $\pi$ is injective. 
	
As $G$ is $\sigma$-compact, $P$ is $\sigma$-compact, and additionally, $H$ is $\sigma$ compact since $P$ has dense image in $H$. We may thus find $K\normal P$ and $R\normal H$ compact subgroups so that $P/K$ and $H/R$ are second countable. The subgroup $\pi(K)$ is then a compact normal subgroup of $H$ with $H/\pi(K)R$ second countable. 

The induced map $\tilde{\pi}:P/K\pi^{-1}(R)\rightarrow H/\pi(K)R$ is an injective, continuous homomorphism with dense normal image. Appealing to Proposition~\ref{Es E}, we have $P/K\pi^{-1}(R)\in \E$. Theorem~\ref{thm:closure_main} thus implies $H/\pi(K)R\in \E$, and as $\Es$ is closed under group extensions, we conclude that $H\in \Es$, verifying the lemma.
\end{proof}

\begin{proof}[Proof of Proposition~\rm\ref{prop:Ys}]
(a) Let $G$ be the directed union of $(O_i)_{i\in I}$ with $O_i\sleq _oG$ and $O_i \in \Ys$ for each $i\in I$. For each $i$, we have $G^\circ = O_i^\circ\sleq O_i$ since $O_i$ is open. So $G/G^\circ=\bigcup_{i\in I}O_i/O_i^\circ$, and therefore, $G/G^\circ \in \Es$.

\medskip \noindent
(b) Given $H \sleq G$, we have that $H/H\cap G^{\circ}$ embeds continuously into $G/G^{\circ}$. By Theorem~\rm\ref{thm:es_closure}, this implies $H/H\cap G^{\circ} \in \Es$. On the other hand, $ H\cap G^\circ\sleq G^\circ$ with $G^\circ$ a connected locally compact group. By the solution to Hilbert's fifth problem, $G^\circ$ is pro-Lie. Let $K\trianglelefteq G^\circ$ be compact such that $G^\circ/K$ is a Lie group. In view of Cartan's theorem, see \cite[LG 5.42]{Se64}, $(H\cap G^\circ)K/K$ is a Lie group. Putting $J:=K\cap H\cap G^\circ$, we have
\[
H^\circ J/J=\left(H\cap G^\circ /J\right)^\circ \sleq_o(H\cap G^\circ)/J,
\]
so $(H\cap G^\circ)/ H^\circ J$ is discrete. Since we may find such $J$ inside arbitrarily small neighborhoods of $1$, the group $(H\cap G^\circ)/ H^\circ$ is residually discrete. Via Theorem~\rm\ref{CM_SIN}, we infer $(H\cap G^\circ)/H^{\circ}$ is a directed union of SIN groups and so in $\Es$. Since $\Es$ is closed under group extension, $H/H^\circ\in \Es$, whereby $H \in \Ys$ as desired.

\medskip \noindent
(c) Let $H \trianglelefteq G$ be such that $H$ and $G/H$ both belong to $\Ys$. To show $G \in \Ys$, we may reduce to $G$ compactly generated. Indeed, if $O\sleq_oG $ is compactly generated, then
\[
H^\circ=(O\cap H)^\circ\sleq O\cap H\trianglelefteq O
\]
and $(O\cap H)/H^\circ\sleq_o H/H^\circ$. We thus deduce that $(O\cap H)/H^\circ\in \Es$, whereby $O\cap H \in \Ys$. Similarly, 
$O/(O\cap H)\in \Ys$. The group $O$ is thus a group extension of groups in $\Ys$. Therefore, if the claim holds for all compactly generated groups, then $O\in \Ys$, and since $\Ys$ is closed under directed unions, $G\in \Ys$. 

Suppose $G$ is compactly generated, put $\tilde{G}:=G/(H\cap G^\circ)$, and let $\pi:G\rightarrow \tilde{G}$ be the usual projection. Since $H^\circ\sleq H\cap G^\circ$, $\tilde{H}:=\pi(H)$ is a quotient of $H/H^\circ$ and, via Theorem~\rm\ref{thm:es_closure}, a member of $\Es$. Setting $\tilde{G^\circ}:=\pi(G^\circ)$, we apply Lemma~\rm\ref{quot lem} to conclude $\overline{\tilde{H}\tilde{G^\circ}}/\tilde{G^\circ}\in \Es$. Since 
\[
\overline{\tilde{H}\tilde{G^\circ}}/{\tilde G^\circ}\simeq \overline{HG^\circ}/G^\circ,
\]
we indeed have that $\overline{HG^\circ}/G^\circ\in \Es$.  

The connected component of any locally compact group coincides with the intersection of all its open subgroups. This implies $\overline{HG^\circ}/H=(G/H)^\circ$, so 
\[
(G/G^\circ)/(\overline{HG^\circ}/G^\circ)\simeq G/\overline{HG^\circ} \simeq (G/H)/(\overline{HG^\circ}/H)\in \Es.
\]
We conclude $G/G^\circ$ is an extension of a group in $\Es$ by another group in $\Es$ and thus a member of $\Es$. Therefore, $G\in \Ys$.

\medskip \noindent
(d) Let $G \in \Ys$ and $N \trianglelefteq G$. As noticed above, we have $\overline{NG^\circ}/N = (G/N)^\circ$. Therefore, in order to show that $G/N \in \Ys$, it suffices to show that $G/\overline{G^\circ N} \in \Es$. The latter is isomorphic to a quotient of $G/G^\circ  \in \Es$, so the desired conclusion follows from Theorem~\rm\ref{thm:es_closure}.
\end{proof}

\begin{cor} \label{cor:Y=Ys}
$\Y = \Ys \subseteq \X$.
\end{cor}
\begin{proof} 
We have already observed that $\Ys \subseteq \Y$. Since $\Ys$ is stable under group extensions and directed unions of open subgroups by Proposition~\rm\ref{prop:Ys}, the reverse inclusion follows. Taking $G\in \Ys$, we see $G/G^{\circ}\in \Es\subseteq \X$ and, via Theorem~\rm\ref{thm:closure}, $G\in \X$. 
\end{proof}

\begin{proof}[Proof of Theorem~\rm\ref{thm:Y}]
We have $\Y \subseteq \X$ by Corollary~\rm\ref{cor:Y=Ys}. The permanence properties of $\Y$ follow from Proposition~\rm\ref{prop:Ys}.
\end{proof}

\begin{rmk}\label{rem:counterexample}
The results of this note give new information concerning locally compact groups $G\notin\X$. As remarked in \cite{CM14}, we may take $G\notin \X$ to be a compactly generated \tdlc group. Applying Theorem~\rm\ref{KK}, we have a compact $K\trianglelefteq G$ such that $G/K$ is second countable. Theorem~\rm\ref{thm:closure} implies $G/K$ must also lie outside of $\X$. We may thus take $G\notin \X$ to be a compactly generated \tdlcsc group. 

By \cite[Theorem 7.8]{W_1_14}, a \tdlcsc group $G$ admits a unique maximal closed normal elementary subgroup, denoted $\Rad_{\E}(G)$ and called the \textbf{elementary radical} of $G$. Suppose $G$ is a \tdlcsc group outside of $\X$ and fix $U\in \U(G)$. By Theorem~\rm\ref{key lem}, $\Rad_{\E}(G)\in \X$, and since $U\Rad_{\E}(G)\in \E$, we further have $U\Rad_{\E}(G)\in \X$. In view of Lemma~\rm\ref{lem:grp_ext}, it must be the case $G/\Rad_{\E}(G)\notin \X$. We may thus suppose $G\notin \X$ has trivial elementary radical, and via \cite[Corollary 9.12]{W_1_14}, $G$ is $[A]$-semisimple. The definition and a discussion of $[A]$-semisimple groups may be found in \cite{CRW_1_13}. Here we merely recall that $[A]$-semisimple groups have a canonical action on a lattice and, in many cases, a non-trivial boolean algebra \cite{CRW_1_13}. 

\end{rmk}
\bibliography{biblio1}
\bibliographystyle{bibgen}
\end{document}